\documentclass[journal,onecolumn,12 pt]{IEEEtran}

\renewcommand{\baselinestretch}{1.5}

\usepackage{cite}

\ifCLASSINFOpdf
 
\else
  
\fi

\usepackage{graphicx} % for pdf, bitmapped graphics files
\usepackage{epsfig} % for postscript graphics files
\usepackage{times} % assumes new font selection scheme installed
\usepackage{amsmath} % assumes amsmath package installed
\usepackage{amssymb} 
\usepackage{multicol}
\usepackage{url}
\newtheorem{theorem}{Theorem}

\newtheorem{assumption}[theorem]{Assumption}
\usepackage{color}
\usepackage{mathrsfs}
\usepackage{multirow}
\usepackage{caption}
\usepackage{subcaption}
\usepackage{wrapfig}
\usepackage{relsize}

\newtheorem{definition}[theorem]{Definition}

\newtheorem{lemma}[theorem]{Lemma}
\newtheorem{notation}[theorem]{Notation}

{ % these settings define remarks and example as parts
\newtheorem{remark}[theorem]{Remark}

}
\newenvironment{proof}[1][Proof]{\textbf{#1.} }{\ \hspace*{\fill} \rule{0.5em}{0.5em}}

\newcommand{\trace}{\text{tr}}

\newcommand{\bi}{\begin{itemize}}
\newcommand{\ei}{\end{itemize}}
\newcommand{\bd}{\begin{displaymath}}
\newcommand{\ed}{\end{displaymath}}
\newcommand{\be}{\begin{eqnarray*}}
\newcommand{\ee}{\end{eqnarray*}}

\usepackage{ifthen}
\newboolean{showcomments}
\setboolean{showcomments}{true}

\newcommand{\highlight}[1]{\ifthenelse{\boolean{showcomments}}
{ \textcolor{red}{#1}}{}}

\begin{document}
\title{
% Stability Analysis and Synthesis of Continuous-time Linear Networked Stochastic Systems}
% Mean Square Stability Analysis and Synthesis of Continuous-time Linear Networked Stochastic Systems}
% Stability Analysis and Control Synthesis of Stochastic Continuous-Time Linear Networked Systems}
Mean Square Stability Analysis of Stochastic Continuous-time Linear Networked  Systems}

\author{ Sai Pushpak, Amit Diwadkar, and Umesh Vaidya% <-this % stops a space
\thanks{Financial support from the National Science Foundation grant CNS-1329915 and ECCS-1150405 is gratefully acknowledged. Sai Pushpak and Umesh Vaidya are with the Department of Electrical \& Computer Engineering,
Iowa State University, Ames, IA 50011. Emails: {\tt\small \{pushpak, ugvaidya\}@iastate.edu, amit.diwadkar@gmail.com}}
}
%\thanks{Manuscript received April 19, 2005; revised September 17, 2014.}}

% \vspace{-4.5 cm}

\maketitle

%\vspace{-2.5 cm}

\begin{abstract}
In this technical note, we study the mean square stability-based analysis of stochastic continuous-time linear networked systems. The stochastic uncertainty is assumed to enter multiplicatively in system dynamics through input and output channels of the plant. Necessary and sufficient conditions for mean square exponential stability are expressed in terms of the input-output property of deterministic or nominal system dynamics captured by the {\it mean square} system norm and variance of channel uncertainty. The stability results can also be interpreted as a small gain theorem for continuous-time stochastic systems.  Linear Matrix Inequalities (LMI)-based optimization formulation is provided for the computation of mean square system norm for stability analysis. For a special case of single input channel uncertainty, we also prove a fundamental limitation result that arises in the mean square exponential stabilization of the continuous-time linear system. Overall, the contributions in this work generalize the existing results on stability analysis from discrete-time linear systems to continuous-time linear systems with multiplicative uncertainty. Simulation results are presented for WSCC $9$ bus power system to demonstrate the application of the developed framework.
\end{abstract}
%\begin{IEEEkeywords}
%IEEEtran, journal, \LaTeX, paper, template.
%\end{IEEEkeywords}

\IEEEpeerreviewmaketitle
\section{Introduction}
The problem of stability analysis and control synthesis of systems in the presence of uncertainty has a rich, long history of literature. The literature in this area can be broadly divided into two parts. One part deals with the norm bounds on uncertainty, and the other part considers the uncertainty to be a stochastic random variable. Classical robust control addresses this problem when uncertainty is norm bounded \cite{Dullerud_Paganini,Skogestad_book}. In this technical note, we study the robust control problem for continuous-time linear dynamics, where the uncertainty is modeled as a stochastic random variable. The stochastic uncertainty is assumed parametric and hence enters multiplicatively in the system dynamics. The analysis and control problem with stochastic multiplicative uncertainty has received renewed attention lately as a model for network controlled system with communication uncertainty.

Some of the classical results involving stochastic stability analysis and control problems are presented in \cite{Hasminskii_book}. The work by Wonham \cite{wonham1967optimal} is one of the earliest literature on this topic involving continuous-time dynamics with multiplicative measurement and control noise.
In \cite{willems1971frequency}, frequency domain-based stability criteria for continuous-time LTI system with state-dependent noise is derived. The authors in \cite{mclane1971optimal} study the LQR problem for continuous-time linear systems with state-dependent noise entering only in the state dynamics.
In \cite{willems1976feedback}, mean square exponential stability analysis and static state feedback control design for stochastic systems with state-dependent control noise are studied.  The same authors, using state feedback control, developed robust stabilization results for continuous and discrete-time uncertain LTI systems in \cite{willems1983robust}. 
 
 In \cite{Bouhtouri_stability_radii} and  \cite{Bouhtouri_stability_radii2}, the authors, using state feedback propose an input-output operator approach for characterizing the stability radii and maximizing the stability radii. In \cite{stochastic_berstein}, the author provides a comparison of necessary and sufficient conditions with dynamic and static output feedback controller involving stochastic multiplicative uncertainty, and deterministic norm bounded uncertainty respectively. Bernstein \cite{stochastic_berstein} also provides a comprehensive survey of literature on this topic of stochastic stability analysis and control.
In \cite{state_feedback_multiplicative}, a linear matrix inequality (LMI)-based mean square exponential stability result using static state feedback control is given for continuous-time LTI systems with state-dependent noise. Using input/output operator approach, a small gain theorem for stochastic systems with state-dependent noise only affecting the state dynamics has been developed in \cite{Dragan_small}. In contrast to these references, we develop mean square exponential stability analysis and synthesis results with stochastic multiplicative uncertainty, both at the input and output side of the plant. The problem formulation is general enough to address problems involving not only input-output channel uncertainty but also parametric stochastic uncertainty.

Research activities in the area of network controlled system have lead to the renewed interest in the analysis and design of systems with multiplicative uncertainty \cite{networksystems_foundation_sastry}. In particular, network systems with erasure or time-delay uncertainty in the input or output communication channel can be modeled as a system with multiplicative uncertainty. Issues related to fundamental limitations for stabilization and estimation of networked systems, i.e., largest tolerable  channel uncertainty are addressed in \cite{scl04,networksystems_foundation_sastry,Tatikonda_noisy_comm,martins_elia,tac11,gupta}. Fundamental limitation results are extended to nonlinear systems in \cite{diwadkar2013limitations,vaidya2010limitations}. Similarly, the problem of fundamental limitations in linear and nonlinear consensus networks with stochastic interactions among network components are addressed in \cite{diwadkar2011robust,vaidya2012limitation, diwadkar2014stochastic, Amit_nature,nicola_networkcontrol,diwadkar2014stabilization}. There is also extensive literature on stability analysis of systems involving nonlinear dynamics with multiplicative stochastic uncertainty \cite{stochastic_kristic2,lure_stab_journal}. A small gain theorem for MIMO linear systems with multiplicative noise in the mean square sense is given in \cite{small_gain_discretetime}. In \cite{Bassam_structured_unc}, the author considers the discrete-time system with correlated stochastic uncertainties and develops necessary, sufficient conditions for mean square exponential stability expressed in terms of the spectral radius of input-output linear matrix operator.   
However, all the above results are developed for discrete-time network dynamical systems. The results in this note are inspired from \cite{scl04} and can be viewed as a continuous-time counterpart of the discrete-time results developed in \cite{scl04}. Following \cite{scl04}, we provide a robust control-based framework for the analysis and synthesis of continuous-time linear networked systems with stochastic channel uncertainties. 

The main contributions of this technical note can be stated as follows. We provide a necessary, sufficient condition for mean square exponential stability of the continuous-time linear networked system with input and output channel uncertainties. The necessary, sufficient conditions for mean square exponential stability are stated in the form of a spectral radius involving mean square system norm. We show the mean square system norm introduced in \cite{scl04} for discrete-time system generalizes to the continuous-time setting. LMI-based optimization formulation is proposed for the computation of the mean square system norm. Furthermore, fundamental limitation result for mean square exponential stabilization expressed in terms of the unstable eigenvalues of the open-loop system is derived.
One of the main differences between the discrete-time problem set-up discussed in \cite{scl04} and continuous-time problem set-up is that we assume the plant dynamics to be strictly proper. The assumption is necessary to avoid two white noise processes from multiplying each other when the signal traverse in the feedback loop. Furthermore, by adopting density-based deterministic approach involving Fokker-Planck equation \cite{Lasota_book}, we avoid the technical challenges associated in dealing with stochastic calculus of stochastic differential equations.

The mean square stability analysis framework developed for continuous-time stochastic network dynamical system  is applicable in variety of settings. Some applications include consensus problems with communication channel uncertainty \cite{Bassam_Axelby_paper,Wang_TSP_2013}, distributed optimization over stochastic network \cite{Sai_ICC_optm,Wang_ACC_2016,Wang_IFAC_2014,Wang_ACC_2012}, multi-agent systems with communication uncertainties \cite{Tao_multi_agent_comm_noise,kim2011output}, and power system with communication channel uncertainty or renewable uncertainties \cite{Sai_ACC_wac,Sai_ACC_fragility}.

\section{Preliminaries and Definitions}
\label{prelims}
This section consists of preliminaries and definitions behind the density function based approach for the analysis of stochastic differential equations (SDE's). 
Consider the following linear SDE with stochastic multiplicative uncertainty,
\begin{align}
\dot{x} = & A x + {\sum_{\ell=1}^m} \sigma_{\ell} B_{\ell} x \xi_{\ell}
\label{stochastic_sys}
\end{align}
where $x \in \mathbb{R}^n$, for ${\ell} = 1, \dots, m$, $\xi_{\ell} = \frac{d\Delta_{\ell}}{dt}$ with $\Delta_1, \dots, \Delta_m$ being the standard independent Wiener process (Brownian motion) and  $\sigma_{\ell} > 0$ for ${\ell} = 1, \dots, m$. The standard independent Wiener processes, $\Delta_i(t)$ for every ${\ell} = 1, \dots, m$ satisfy, 
\begin{enumerate}
\item[(i)] $\text{Prob} \{\Delta_{\ell}(0) = 0\} = 1$ 
\item[(ii)] $\{\Delta_{\ell}(t) \}$ is a process with independent increments 
\item[(iii)] $\{\Delta_{\ell}(t) - \Delta_{\ell}(s)\}$ has a Gaussian distribution with $E[\Delta_{\ell}(t) - \Delta_{\ell}(s)] =0$ and $E[(\Delta_{\ell}(t) - \Delta_{\ell}(s))^2] = |t-s|$. 
\end{enumerate}
\begin{notation} 
In the following, $x(t)$ is used  to denote the solution of system (\ref{stochastic_sys}) defined in the sense of It\^o  and notation $x$ is used to describe the states. We refer the readers to \cite[Theorem 11.5.1]{Lasota_book} for technical assumptions leading to existence and uniqueness of solution to SDE (\ref{stochastic_sys}). It is important to emphasize that these assumptions are satisfied by  (\ref{stochastic_sys}). 
\end{notation}
Next, we state the following stability definition for system \eqref{stochastic_sys}. 
\begin{definition}\label{def_mse}[Mean Square Exponentially Stable (MSES)]
System (\ref{stochastic_sys}) is mean square exponentially stable, if there exists positive constants $\beta_1$ and $\beta_2$, such that, 
$$E[ x\textcolor{black}{(t)}^\top x\textcolor{black}{(t)} ]\leq \beta_1 \exp({-\beta_2 t})E[{x(0)}^\top x(0)], \;\; \forall\; x(0)\in \mathbb{R}^n.$$
\end{definition}
We now consider the following SDE with multiplicative as well as additive stochastic uncertainty,  
\begin{align}
\dot{x} = & A x +\textstyle \sum_{{\ell}=1}^m \sigma_{\ell} B_{\ell} x \xi_{\ell} + H \eta
\label{system_additive}
\end{align}
where $x \in \mathbb{R}^n$, $H \in \mathbb{R}^n$, for ${\ell} = 1, \dots, m$, $\xi_{\ell} = \frac{d\Delta_{\ell}}{dt}$, $\eta = \frac{d \zeta}{dt}$ with $\Delta_1, \dots, \Delta_m, \zeta$ being the standard independent Wiener process (Brownian motion). It is assumed that the standard Wiener process,  $\zeta$ is uncorrelated with the processes, $\Delta_1,\ldots, \Delta_m$. We now define the notion of bounded moment stability for system (\ref{system_additive}).
\begin{definition}\label{def_smb}[Second Moment Bounded]
System (\ref{system_additive}) is said to be second moment bounded if there exists a positive constant $\beta$, such that,
$$\lim_{t\to \infty} E[x\textcolor{black}{(t)}^\top x\textcolor{black}{(t)}]\leq \beta, \;  \forall \ x(0)\in \mathbb{R}^n.$$
\end{definition}
The result establishing the relation between the mean square exponential stability and second moment boundedness is discussed in the later part of this section. 

Instead of analyzing the individual trajectories, $x(t)$, we adopt density-based approach as proposed in \cite{Lasota_book} towards the analysis of stochastic system (\ref{system_additive}). In particular, the density function $\psi(x,t)$ for the stochastic process $x(t)$ satisfies
$$\text{Prob}\{x(t) \in B\} =  \int_{B} \psi(z,t) dz$$ for any set $B\subset\mathbb{R}^n$. The density function, $\psi(x,t)$  is obtained as a solution of a linear partial differential equation, known as the Fokker-Planck (FP) equation, also called the Kolmogorov forward equation \cite[Theorem 11.6.1]{Lasota_book}. The FP equation is  defined as follows
\begin{align}
{\frac{\partial \psi}{\partial t}} = & {\frac{1}{2} \sum_{i,j = 1}^n \frac{\partial^2}{\partial x_i \partial x_j}\left( \sum_{{\ell}=1}^m \sigma_{\ell}^2 (b_{\ell}^ix)(b_{\ell}^jx) + h_i h_j\right) \psi} { - \sum_{i=1}^n \frac{\partial} {\partial x_i} (a_i x) \psi, \;\; t > 0, \; x \in \mathbb{R}^n,}
\label{fp_eq}
\end{align}
where $a_i,b_{\ell}^i $ are the $i^{th}$ rows of $A, B_{\ell}$ respectively and $h_i$ is the $i^{th}$ entry of $H$ in (\ref{system_additive}). 
\begin{remark}
The coefficients,  $\sum_{{\ell}=1}^m \sigma_{\ell}^2 (b_{\ell}^i x) (b_{\ell}^j x) + h_i h_j$, $a_i x$ satisfy the uniform parabolicity condition and hence they are regular \cite[Definition 11.7.2]{Lasota_book}.  Based on these properties of coefficients, the solution of FP equation satisfies following bounds \cite[Theorem 11.7.1]{Lasota_book}, $$\vert \psi \vert, \vert\frac{\partial \psi}{\partial t}\vert, \vert \frac{\partial \psi}{\partial x_i} \vert, \vert\frac{\partial^2 \psi}{\partial x_i \partial x_j}\vert \leq \bar K t^{-(n+2)/2} \exp(-\frac{1}{2}\bar \delta \vert x \vert^2/t),$$ where $\bar K$ and $\bar \delta$ are positive constants and are function of bounds that appear in the uniform parabolicity condition and bounds on the initial density function, $\psi(x,0)$.
\label{remark_fp_eq}
\end{remark}
These bounds on $\psi$ and its derivatives allow us to multiply the FP equation  \eqref{fp_eq} with any increasing function that increases more slowly than $\exp(-\frac{1}{2}\vert x \vert^2)$. The resultant function is decreasing, and we can integrate term by term to compute moments of $\psi(x,t)$. It is known, for the case of a linear system driven by additive white noise process, if the initial density function, $\psi(x,0)$, is Gaussian, then $\psi(x,t)$ remains Gaussian for all future time $t$. Hence, for linear systems with additive white noise forcing, the infinite dimensional FP equation can be replaced with the finite dimensional equation for the evolution of the mean and covariance. In the following lemma, we show the covariance evolution for the system (\ref{system_additive}) with multiplicative noise is closed and does not depend upon higher order moments. 

\begin{lemma} Let the covariance matrix, $\bar Q(t)=E[x(t)x(t)^\top|\psi]:=\int_{\mathbb{R}^n}xx^\top \psi(x,t)dx,$ and $\bar{Q}(0): = \bar{Q}_0 < \infty,$
then $\bar Q(t)$  satisfies the following matrix differential equation (MDE) for system \eqref{system_additive} 
\begin{align}
\dot {\bar Q}=\bar QA^\top+A\bar Q+\textstyle \sum_{{\ell}=1}^m \sigma_{\ell}^2 B_{\ell} \bar QB_{\ell}^\top + H H^\top.
\label{cov_mde_add}
\end{align}
%and the covariance equation for system (\ref{stochastic_sys}) satisfies
%\begin{eqnarray}
%\dot{{Q}}={Q}A^\top + A {Q} + \sum_{k=1}^p \sigma_k^2 B_k {Q}B_k^\top.
%\label{cov_mde}
%\end{eqnarray}
\label{lemma_cov}
\end{lemma}
% 
%\vspace{-0.5 cm}
\begin{proof}
Consider the quadratic function, $V(x) = x^{\top} P x$, for any $P = P^{\top} > 0$ that is increasing. 
Then, $$E[V | \psi] := \int_{\mathbb{R}^n} V(x) \psi(x,t) dx.$$ Taking the time derivative on both sides and after simplification, we obtain \cite[Theorem 11.9.1]{Lasota_book}
\begin{align}
\frac{d E[V|\psi]}{dt}  = & \int_{\mathbb{R}^n} \left\{\frac{1}{2} \sum_{i,j = 1}^n \Bigg[\sum_{{\ell}=1}^m \sigma_{\ell}^2 (b_{\ell}^i x) (b_{\ell}^j x) + h_i h_j\Bigg] \frac{\partial^2 V}{\partial x_i \partial x_j}  + \sum_{i=1}^n (a_i x) \frac{\partial V}{\partial x_i}\right\} \psi(x,t) dx =E[\mathcal{L} V|\psi].
\label{eq_deriv_oper}
\end{align}
In Eq. \eqref{eq_deriv_oper}, the term $\mathcal{L}V$ is given by 
\begin{equation}
\mathcal{L}V =  x^{\top} \Big(A^{\top} P + P A + \sum_{{\ell}=1}^{m} \sigma_{\ell}^2 B_{\ell}^\top P B_{\ell}
\Big) x + H^{\top} P H.
\label{eq_diff_oper}
\end{equation} 
The time derivative of $E[V|\psi]$ is obtained by doing integration by parts where we make use of Remark \ref{remark_fp_eq}. In particular, we make use of the fact that the products, $\psi V, \frac{\partial \psi}{\partial x_i} V, \psi \frac{\partial V}{\partial x_i}$ vanish exponentially as $\vert x \vert \to \infty$ and hence, the higher order moments become zero. 
By substituting Eq. \eqref{eq_diff_oper} in Eq. \eqref{eq_deriv_oper}, and using the linearity of trace, expectation and commutativity inside trace, we obtain, 
$$ {d (\trace(E[xx^{\top}|\psi] P))}/{dt} = \trace\Big(\big(A^{\top} P + P A +\textstyle \sum_{{\ell}=1}^m \sigma_{\ell}^2 B_{\ell}^{\top} P B_{\ell} \big) E[x x^{\top}|\psi] + H H^{\top}P \Big).$$
By definition of expectation, $E[xx^{\top}|\psi] = \bar{Q}$, we have, 
$$ \trace(\dot{\bar{Q}} P) =  \trace\Big((\bar{Q}A^\top+A\bar{Q}+\textstyle \sum_{{\ell}=1}^m \sigma_{\ell}^2 B_{\ell} \bar{Q}B_{\ell}^\top + H H^{\top})P \Big). $$
This can be rewritten in terms of an inner product as
\[ \langle \dot{\bar{Q}} -(\bar{Q}A^\top+A\bar{Q}+\textstyle \sum_{{\ell}=1}^m \sigma_{\ell}^2 B_{\ell} \bar{Q} B_{\ell}^\top + H H^{\top}), P   \rangle = 0. \]
Since, $P > 0$, $\dot{\bar{Q}} = \bar{Q} A^\top + A\bar{Q} +\sum_{{\ell}=1}^m \sigma_{\ell}^2 B_{\ell} \bar{Q} B_{\ell}^\top + H H^{\top}. $ Furthermore, for $H = 0$, we obtain the covariance propagation equation for the system \eqref{stochastic_sys} without additive noise. 
\end{proof}

The ensuing result relates the stochastic stability of systems \eqref{stochastic_sys} and \eqref{system_additive}. 
\begin{lemma} The system (\ref{stochastic_sys}) is mean square exponentially stable if and only if system (\ref{system_additive}) is second moment bounded. 
\end{lemma}
\begin{proof}
Let $\phi: \mathbb{R}^{n \times n} \to \mathbb{R}^{n^2}$ be a bijective operator\cite[Chapter 2]{Costa_book} which converts a matrix into a column vector. Then, applying the operator, $\phi$ on both sides of MDE's,  Eq. \eqref{stochastic_sys} and Eq. \eqref{system_additive}, they can be written as linear vector differential equations. 
\begin{eqnarray}
\dot{\vartheta} = & \mathscr{A} \vartheta,
\label{cov_lde} \\
\dot{\bar \vartheta} = & \mathscr{A} \bar \vartheta + \mathscr{B},
\label{cov_lde_add}
\end{eqnarray}
where $\vartheta = \phi( Q), \bar \vartheta = \phi(\bar Q)$, $\mathscr{B} = (G \otimes G) \phi(I)\in \mathbb{R}^{n^2} \text{and}$ 
 $\mathscr{A} = A \oplus A + \sum_{k=1}^p \sigma_k^2 (B_k \otimes B_k)\in\mathbb{R}^{n^2\times n^2},$
where $I$ is the identity matrix of size $n \times n$ and $\otimes$ denotes the Kronecker product, $\oplus$ is the Kronecker sum.

\textit{Necessity}:
The mean square exponential stability of system \eqref{stochastic_sys} yields stability of system \eqref{cov_lde}, that is, $\mathscr{A}$ is Hurwitz. Since  $\mathscr{A}$ is Hurwitz, the steady state value of $\bar \vartheta$ is given by 
\[\lim_{t \to \infty} \bar \vartheta(t) =    \lim_{t \to \infty}  \phi(\bar Q(t)) =  -\mathscr{A}^{-1} \mathscr{B}.\]
Now, taking the inverse $ \phi$ operator, we obtain,
\[ \lim_{t \to \infty} E[x(t)x(t)^{\top}|\psi] =  - \phi^{-1}(\mathscr{A}^{-1} \mathscr{B}),\]
where $ \phi^{-1}(\mathscr{A}^{-1} \mathscr{B})$ is finite. Therefore, system \eqref{system_additive} is second moment bounded.

\textit{Sufficiency:} If system \eqref{system_additive} is second moment stable, then
$ \lim_{t \to \infty} \bar Q(t) $ is a finite value. 
Then, we have
\begin{align*}
\lim_{t \to \infty} \bar Q(t) = \lim_{t \to \infty} \phi^{-1}(\vartheta(t)) = \lim_{t \to \infty} \phi^{-1}(e^{\mathscr{A}t} \vartheta(0)+(e^{\mathscr{A}t}-I) \mathscr{A}^{-1} \mathscr{B}), 
\end{align*}
where $I$ is the identity matrix. 
The limit on the right-hand side is finite, if and only if $\mathscr A$ is Hurwitz, which implies system (\ref{cov_lde}) is stable and hence system (\ref{stochastic_sys}) is mean square exponentially stable.
\end{proof}
\section{Stochastic Uncertainty Modeling} 
\label{stochasticity_modeling}
In this section, we discuss how the stochastic uncertainty enters into the system dynamics. The problem set-up (as shown in Fig. \ref{fig_ip_op_unc2}) follows closely with the one used in \cite{scl04} for mean square exponential stability analysis of a discrete-time network. 
The dynamics of the plant are described by 
\begin{equation}
\small
\mathbb{P}:\left\{
\begin{array}{lcl}
\dot x_p &=&A_p x_p+B_p u_p\\
y_p&=&C_p x_p
\end{array}\right.,
\label{plant}
\end{equation}
where $x_p \in \mathbb{R}^n, u_p\in \mathbb{R}^d$, and $y_p\in \mathbb{R}^q$ are the plant state, input, and output, respectively. The  state space model for the plant is assumed to be stabilizable, detectable, and strictly proper.
Similarly, the controller dynamics are assumed to be strictly proper with the following state space model. 
\begin{equation}
\small
\mathbb{K}: \left\{
\begin{array}{lcl}
\dot x_k&=&A_k x_k + B_k y_k\\
u_k &=&C_k x_k
\end{array}\right.,
\label{controller}
\end{equation}
where $x_k\in \mathbb{R}^n, y_k\in \mathbb{R}^q$, and $u_k\in \mathbb{R}^d$.
The assumption on the controller dynamics being strictly proper is essential, since it allows us to study the case where the uncertainties enter at both input and output channels (refer subsection \ref{sec_proper_controller} for more explanation). If the uncertainty  enters only at the input or the output channel, then one can consider the controller dynamics which is not strictly proper \cite{Sai_cdc2015}. 
\begin{figure}[h]
 \begin{center}
\includegraphics[scale=1.3]{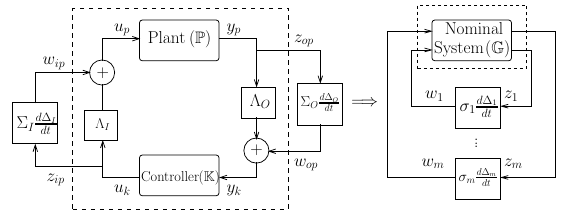}
\caption{a) MIMO plant and controller with stochastic uncertainty in the input and output channels b) MIMO nominal system with stochastic uncertainty in the feedback}
\label{fig_ip_op_unc2}
 \end{center}
\end{figure}

%\vspace{-0.4 cm}
The output of the plant before reaching the controller is affected by stochastic uncertainty and is given by 
$y_k = \Xi_O y_p.$
Similarly, the input to the plant from controller is affected by stochastic uncertainty which is 
$u_p = \Xi_I u_k.$
The output ($\Xi_O$) and input ($\Xi_I$) channel uncertainties can be separated into mean and zero mean part as shown below. 
\[\Xi_O = \Lambda_O + \Sigma_O \frac{d\Delta_O}{dt}, \quad  \Xi_I = \Lambda_I + \Sigma_I \frac{d\Delta_I}{dt},\]
where $\Lambda_{I(O)}$ are the mean part of the input (output) channel uncertainty, $\Sigma_{I(O)}$ are the standard deviation of the input (output) channel uncertainty, and $\Delta_{I(O)}$ denotes the vector valued independent Wiener processes. The corresponding matrices are defined as, 
\begin{align*}
& \mathsmaller{ \Lambda_O:= {\rm diag} (\lambda_O^1,\ldots,\lambda_O^q),  \Lambda_I:= {\rm diag} (\lambda_I^1,\ldots,\lambda_I^d),} \\
& \mathsmaller{ \Sigma_O:= {\rm diag} (\sigma_O^1,\ldots, \sigma_O^q),  \Sigma_I := {\rm diag}(\sigma_I^1,\ldots, \sigma_I^d),}  \\
& \mathsmaller{ \frac{d\Delta_O}{dt}:= {\rm diag} \Big(\frac{d\Delta_O^1}{dt},\ldots, \frac{d\Delta_O^q}{dt}\Big), \frac{d\Delta_I}{dt}:= {\rm diag} \Big(\frac{d\Delta_I^1}{dt},\ldots, \frac{d\Delta_I^d}{dt}\Big).} 
\end{align*}
Both the input and output channel uncertainties are assumed to be uncorrelated. 

Figure \ref{fig_ip_op_unc2}a consists of MIMO plant and controller interacting through uncertain inputs and outputs. The nominal part of this stochastic closed-loop system ($\mathbb{G}$) is the deterministic part of the stochastic closed-loop system which consists of MIMO plant ($\mathbb{P}$), controller ($\mathbb{K}$) and the mean part of the uncertainties ($\Lambda_O, \Lambda_I$). This nominal part, denoted by  $\mathbb{G}={\cal F}(\mathbb{P},\mathbb{K})$, which is essentially the feedback interconnection of plant and controller interacting through the mean part of uncertain channels and is shown inside the dotted line in Fig. \ref{fig_ip_op_unc2}a. Now, the nominal system in Fig. \ref{fig_ip_op_unc2}a, interacts with the stochastic uncertainty via the disturbance ($w_{op} \in \mathbb{R}^q$ and $w_{ip} \in \mathbb{R}^d$) and control signals ($z_{op} \in \mathbb{R}^q$ and $z_{ip} \in \mathbb{R}^d$). The disturbance and control signals are defined as follows, $w_{op}=\Sigma_O \frac{d\Delta_O}{dt} z_{op}, w_{ip}=\Sigma_I \frac{d\Delta_I}{dt} z_{ip},$ and $z_{op}=y_p=C_p x_p, z_{ip}=u_k=C_k x_k$.
The nominal system has the following state space form. 
\begin{equation} \mathbb{G}:\left\{
\begin{array}{lcl}
\dot x&=&A x+B w \\
z&=&Cx
\end{array}\right.,
\label{eq_plant_dynamics}
\end{equation}
where $x=\begin{bmatrix} x_p^\top & x_k^\top\end{bmatrix}^\top\in \mathbb{R}^{2n}$, $z=\begin{bmatrix}z_{op}^\top & z_{ip}^\top\end{bmatrix}^{\top}\in \mathbb{R}^m$, $w=\begin{bmatrix}w_{op}^\top & w_{ip}^\top\end{bmatrix}^{\top}\in \mathbb{R}^m$, $C = {\rm diag} (C_p, C_k)$, 
\[\mathsmaller{A=\begin{bmatrix}A_p& B_p \Lambda_I C_k \\B_k \Lambda_O C_p& A_c \end{bmatrix},\;\;\; B= \begin{bmatrix}0& B_p\\B_k&0\end{bmatrix}.}\]

Finally, this nominal system, $\mathbb{G}$ interacting with stochastic uncertainty, $\frac{d\Delta}{dt}$ can be written in the standard robust control form, $\mathcal{F}(\mathbb{G},\frac{d\Delta}{dt})$ as shown in Fig. \ref{fig_ip_op_unc2}b, where the stochastic uncertainty, $\frac{d\Delta}{dt} = \text{diag}\Big(\sigma_1 \frac{d\Delta_1}{dt}, \dots, \sigma_m \frac{d \Delta_m}{dt}\Big)$. We show the nominal part of the stochastic closed-loop system inside the dotted line in Fig. \ref{fig_ip_op_unc2}b and for clarity, the individual uncertain channels are shown in feedback. Thus, the resultant stochastic closed-loop system (stochastic MIMO system) has number of feedback connections equal to the number of uncertainties. We now re-enumerate the input and output uncertainties, $\Delta^1_I,\ldots, \Delta^d_I$, $\Delta^1_O,\ldots, \Delta^q_O$ as $\Delta_1, \Delta_2, \dots, \Delta_m$, where $m = d+q$. The closed-loop system $\mathcal{F}(\mathbb{G},\frac{d\Delta}{dt})$ has the following sate space form.
\begin{align}
\begin{array}{lcl}
\dot x & = & Ax+Bw \\
z & = & Cx\\
w & = &\frac{d\Delta}{dt} z:=\begin{bmatrix}\Sigma_O \frac{d\Delta_O}{dt}&0\\0&\Sigma_I \frac{d\Delta_I}{dt}\end{bmatrix} z
\end{array}
\label{eq_system_unc}
\end{align}
where $\frac{d\Delta}{dt} = \text{diag}\Big(\sigma_1 \frac{d\Delta_1}{dt}, \dots, \sigma_m \frac{d \Delta_m}{dt}\Big)$.
%%%
\begin{remark} Although we arrive at system (\ref{eq_system_unc}) given in standard robust control form with input and output channel uncertainties, the framework is general enough to model stochastic parametric uncertainty in system plant, $A_p$, matrix.  
\end{remark}
Before we conclude this section, we show mathematically that choosing a proper controller will lead to multiplication of white noise processes which many not be a white noise process. 

\subsection{Stochastic uncertainty modeling with proper controller}
\label{sec_proper_controller}
The choice of strictly proper controller is necessary in our formulation to model the stochastic uncertainty in both input and output channels. In this work, the stochastic uncertainty entering the input and output channel is assumed to be uncorrelated and these stochastic uncertainties (white noise) in the feedback loop will multiply when they traverse around the loop. While the multiplication of two white noise process is well defined, the resulting process might not be a white noise and hence the FP equation for the evolution of density cannot be defined. 

In the following, we now show mathematically the choice of proper controller (but NOT strictly proper) will lead to multiplication of white noise process in the closed-loop system. 

Consider the state space of a strictly proper plant:
\begin{equation*}\mathbb{P}:\left\{
\begin{array}{lcl}
\dot x_p &=&A_p x_p+B_p u_p\\
y_p&=&C_p x_p
\end{array}\right.,
\end{equation*}
and the state space of a proper controller:
\begin{equation*}\mathbb{K}: \left\{
\begin{array}{lcl}
\dot x_k&=&A_k x_k + B_k y_k\\
u_k &=&C_k x_k + D_k y_k
\end{array}\right.,
\end{equation*}
and this feedback interconnection of plant and controller are connected through the uncertain channels as shown in Fig. \ref{fig1_example}.
\begin{figure}[h]
\begin{center}
\includegraphics[scale=1.5]{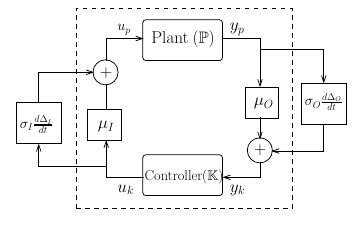}
\caption{Plant and controller with uncertainty in feedback} 
\label{fig1_example}
\end{center}
\end{figure}

The input to the controller and the input to the plant are given by:
\begin{align*}
u_p = & \left(\mu_I + \sigma_I \frac{d\Delta_I}{dt} \right) u_k, \\
y_k = & \left(\mu_O + \sigma_O \frac{d\Delta_O}{dt}\right) y_p,
\end{align*}
where $\xi_O = \frac{d\Delta_O}{dt}, \xi_I = \frac{d\Delta_I}{dt}$ with $\xi_O, \xi_I$ being the white noise processes and $\Delta_O, \Delta_I$ are the independent standard Wiener processes. Now, the closed-loop system is given by 
\begin{align*}
\dot{x}_p = & (A_p + \mu_I \mu_O B_p D_k C_p) x_p + \mu_I B_p C_k x_k + \mu_I \sigma_O B_p D_k C_p x_p \xi_O + \sigma_I B_p C_k x_k \xi_I \\
& + \sigma_I \mu_O B_p D_k C_p x_p \xi_I + \textcolor{red}{\sigma_I \sigma_O B_p D_k C_p x_p \xi_O \xi_I} \\
\dot{x}_k = & A_k x_k + B_k (\mu_O+\sigma_O \xi_O) C_p x_p 
\end{align*}

Notice that, the white noise processes are multiplying in the above closed-loop system (the term marked in red color). This multiplication can be avoided when we consider strictly proper controller or uncertainty in either input or output channels. 
If the uncertainty is assumed either at the input side or output side (but not both side) of the plant, then we can consider the controller to be a proper system. Stability analysis and controller synthesis results for the case with uncertainty either at the input or at the output case without strictly proper assumption on the controller dynamics have been shown in \cite{Sai_cdc2015}.

In the next section, we derive necessary and sufficient conditions for mean square exponential stability of system \eqref{eq_system_unc}. 
%
% \vspace{-0.8 cm}
\section{ Mean Square Stability Analysis}
\label{stability}
We first extend the notion of mean square norm for discrete-time system from \cite{scl04} to continuous-time system. The mean square norm will be used to analyze the mean square stability of system \eqref{eq_system_unc} i.e., the feedback interconnection $\mathcal{F}(\mathbb{G},\frac{d\Delta}{dt})$. However, the norm itself is defined for nominal system $\mathbb{G}$ with multiple inputs and outputs.

\begin{definition}\label{def_msn}[Mean Square Norm] The mean square norm for nominal system, $\mathbb{G}$, is defined as follows. 
\begin{equation*}
{\parallel \mathbb{G} \parallel_{MS}=\max_{i=1,\ldots, m}\sqrt{\sum_{j=1}^m \parallel G_{ij}\parallel^2_2},}
\end{equation*} 
where the system, $G_{ij}$ denotes the transfer function of the nominal system corresponding to the input $j$ and output $i$ and $\parallel G_{ij} \parallel_2$ denotes the standard $\mathcal{H}_2$ norm.
\end{definition}
\begin{remark}
In the definition of mean square norm given above, number of inputs and outputs to the nominal system depend on number of uncertainties in the input and output channels. For example, in Fig. 1, there are $d$ inputs, $q$ outputs respectively and hence, there are $d+q:=m$ feedback channels in the stochastic closed-loop system shown in Fig. \ref{fig_ip_op_unc2}b.
\end{remark}
The stochastic interconnected system \eqref{eq_system_unc} can be written as system \eqref{stochastic_sys} for which, the mean square exponential stability given in Definition \ref{def_mse} applies. We make the following assumption on the feedback interconnected system \eqref{eq_system_unc}. 
%
% \vspace{1.5 cm}
\begin{assumption}
\label{assumptions}
\begin{enumerate}
\item [(a)] The deterministic system \eqref{eq_plant_dynamics} denoted by $\mathbb{G}$ is internally stable, that is, $A$ is Hurwitz and moreover, $\mathbb{G}$ is considered to be stabilizable, detectable and strictly proper.
\item[(b)] The initial state of the system $\mathbb{G}$, denoted by $x(0)$ has bounded variance and is independent from $\Delta_i(t)$ for each $i \in \{1,\dots,m\}$. 
\end{enumerate}
\end{assumption}
These assumptions are common in the control literature \cite{scl04, Dullerud_Paganini}. To investigate the stochastic stability of the feedback interconnection ${\cal F}(\mathbb{G},\frac{d\Delta}{dt})$, it is necessary condition, that the system $\mathbb{G}$ is internally stable. The assumption on stabilizability and detectability is required in the design of controller and for the computation of $\mathcal{H}_2$ norm \cite{Dullerud_Paganini}.  The assumption on strictly proper nature of plant and controller is to avoid the product of two white noise processes (which is defined, but the resultant process might not be a white noise), when the loop is closed. Finally, the assumption on initial condition is to avoid the complexity of math due to correlations.

The following theorem provides necessary and sufficient conditions for the mean square exponential stability of the interconnected system \eqref{eq_system_unc}. 
\begin{theorem}
\label{thm_mss_lyap_eq}
Under Assumption \ref{assumptions}, the feedback interconnected system \eqref{eq_system_unc} shown in Fig. \ref{fig_ip_op_unc2}b is mean square exponentially stable, if and only if, there exists a $P > 0$, such that, it satisfies
\begin{equation}
{A^{\top} P + P A +  \sum_{{\ell}=1}^{m}  \sigma_{\ell}^2 C_{\ell}^{\top} B_{\ell}^{\top} P B_{\ell} C_{\ell} < 0.}
\label{eq_mss_lyap}
\end{equation}
\label{thm_mss_lyap}
\end{theorem}
%
%\vspace{-0.5 cm}
\begin{proof}
\textit{Sufficiency:} The covariance propagation equation for the feedback interconnected system with uncertainty is
\begin{align}
{\dot{Q}(t)} = & {Q(t)A^\top+AQ(t)+ \sum_{{\ell}=1}^m \sigma_{\ell}^2 B_{\ell} C_{\ell} Q(t) C_{\ell}^{\top} B_{\ell}^\top.}
\label{eq_cov_MIMO}
\end{align}
This covariance propagation equation is a matrix differential equation and follows from Lemma \ref{lemma_cov}. To achieve mean square exponentially stable, $Q(t)$ should converge to zero exponentially. To show this, we construct the Lyapunov function $V(Q(t)) = \trace(Q(t)P)$, where $P > 0$. Then, $$\dot{V}(Q(t)) =
 {\trace\Big((Q(t)A^\top+AQ(t) + \sum_{{\ell}=1}^m \sigma_{\ell}^2 B_{\ell} C_{\ell} Q(t) C_{\ell}^{\top} B_{\ell}^\top)P\Big).} $$
Then, we obtain, $\dot{V}(Q(t)) =\trace(-Q(t)M)$ for some positive matrix $M > 0$. Since $M>0$, there exits an 
 $\alpha := \frac{\lambda_{min}(M)}{\lambda_{max}(P)}>0$, 
% $\alpha := {\lambda_{min}(M)}/{\lambda_{max}(P)}>0$, 
such that $\alpha P\leq M$. Therefore, $\dot{V}(Q(t)) \leq -\alpha  V(Q(t))$ and it follows, system \eqref{eq_system_unc} is mean square exponentially stable.

\textit{Necessity:} Let $\phi: \mathbb{R}^{n \times n} \to \mathbb{R}^{n^2}$ be a bijective operator\cite[Chapter 2]{Costa_book} which converts a matrix into a column vector. 
Assume, system \eqref{eq_system_unc} to be mean square exponentially stable. Then, we know, the covariance matrix, $Q(t)$, converges exponentially to zero. This implies, we have a stable evolution for $q(t)=\phi(Q(t))\in \mathbb{R}^{n^2}$ with the following dynamics,
\begin{align}
\dot{q}(t) & = \mathscr{A} q(t),
\label{eq_cov_MIMO_lde}
\end{align}
where $\mathscr{A} = A \oplus A + \sum_{{\ell} = 1}^m \sigma_{\ell}^2 (B_{\ell} C_{\ell} \otimes B_{\ell} C_{\ell})$. 
Stability of system \eqref{eq_cov_MIMO_lde} implies $\mathscr{A}$ is Hurwitz and hence $\mathscr{A}^{\top}$ is also Hurwitz. Therefore, the evolution, $\dot{r}(t)  = \mathscr{A}^{\top} r(t)$ is stable and  satisfies the following matrix differential equation,
\begin{align}
\dot{R}(t) = & A^{\top} R(t) + R(t) A + \sum_{{\ell} =1}^m \sigma_{\ell}^2 C_{\ell}^{\top} B_{\ell}^{\top} R(t) B_{\ell} C_{\ell}
\label{eq_MIMO_Lyap}
\end{align}
which is also stable, where $R(t) = \phi^{-1}(r(t))$. Let $R(0) > 0$, and $R(t)$ denote the solution for Eq. \eqref{eq_MIMO_Lyap}. Since $R(t)$ satisfies the stable first order linear differential equation, the function
$P(t) = \int_{0}^t R(\tau) d\tau $ has a finite value. Integrating on both sides of Eq. \eqref{eq_MIMO_Lyap} and simplifying, we obtain
\begin{align*}
\dot{P}(t) - R(0)  = &
A^{\top} P(t) + P(t) A +  \sum_{{\ell} = 1}^m \sigma_{\ell}^2 C_{\ell}^{\top} B_{\ell}^{\top} P(t) B_{\ell} C_{\ell}.
\end{align*}
Observing that Eq. \eqref{eq_MIMO_Lyap} is stable, as $t \to \infty$, we obtain,
\begin{align*}
A^{\top} P + P A +  \sum_{{\ell}=1}^m \sigma_{\ell}^2 C_{\ell}^{\top} B_{\ell}^{\top} P B_{\ell} C_{\ell} = -R(0), 
\end{align*}
where $P := \lim_{t \to \infty} P(t)$. The result now follows by noticing that $R(0) > 0$.
\end{proof}

The ensuing result gives an alternative representation of the inequality given in Eq. \eqref{eq_mss_lyap}, which is helpful in writing the LMI-based optimization for computing the mean square norm of system $\mathbb{G}$. The following lemmas, theorems, and their proofs can be viewed as the continuous-time counterpart of the discrete-time results from \cite{scl04}. 

\begin{lemma}
The inequality, Eq. \eqref{eq_mss_lyap} holds if and only if there exists a $\mathcal{Q} > 0$, and $\alpha_{\ell} >0$ for every ${\ell} = 1,2,\dots,m$, such that
\begin{align}
\begin{aligned}
& \mathsmaller{A\mathcal{Q} + \mathcal{Q}A^{\top} +  \sum_{{\ell} = 1}^{m} B_{\ell} \alpha_{\ell} B_{\ell}^{\top} < 0,} \\
& {\alpha_{\ell} > \sigma_{\ell}^2 C_{\ell} \mathcal{Q} C_{\ell}^{\top}, \quad {\ell} = 1,2,\dots, m.}
\end{aligned}
\label{eq_mss_2lmi}
\end{align}
\label{lemma_mss_Q}
\end{lemma}
\vspace{-0.3 cm}
\begin{proof}
The dual inequality equivalent to Eq. \eqref{eq_mss_lyap} is
\begin{align}
A \mathcal{Q} + \mathcal{Q} A^{\top} +  \sum_{{\ell}=1}^m \sigma_{\ell}^2  B_{\ell} C_{\ell} \mathcal{Q} C_{\ell}^{\top} B_{\ell}^{\top} < 0,
\label{eq_MIMO_cov_mde}
\end{align}
where $\mathcal{Q} > 0$.
Observe the straightforward substitution leads to a sufficiency condition. In showing the necessary part, for some matrix $M >0 $, the inequality \eqref{eq_MIMO_cov_mde} can be rewritten as
\begin{align}
A \mathcal{Q} + \mathcal{Q} A^{\top} +  \sum_{{\ell}=1}^m  \sigma_{\ell}^2  B_{\ell} C_{\ell} \mathcal{Q} C_{\ell}^{\top} B_{\ell}^{\top} + M = 0.
\label{eq_1}
\end{align}
Since we have $\sum_{{\ell}=1}^m  B_{\ell} B_{\ell}^{\top} \geq 0$, there exists $\epsilon (M) := \epsilon > 0$, such that $0 \leq \epsilon \textstyle \sum_{{\ell}=1}^m B_{\ell} B_{\ell}^{\top} < M.$
%Since, it follows that, $G$ is controllable, $\sum_{i=1}^m  B_i B_i^{\top} > 0$.
Using this in Eq. \eqref{eq_1}, 
\begin{align*}
&A \mathcal{Q} + \mathcal{Q} A^{\top} +  \sum_{{\ell}=1}^m  \sigma_{{\ell}}^2  B_{\ell} C_{\ell} \mathcal{Q} C_{\ell}^{\top} B_{\ell}^{\top} + \epsilon  \sum_{{\ell}=1}^m  B_{\ell} B_{\ell}^{\top} < 0, \\
%& + \epsilon \sum_{i=1}^m  B_i B_i^{\top} < 0 \\
\Rightarrow &A \mathcal{Q} + \mathcal{Q} A^{\top} +  \sum_{{\ell}=1}^m \left( \epsilon + \tilde \sigma_{{\ell}}^2 C_{\ell} \mathcal{Q} C_{\ell}^{\top}\right) B_{\ell} B_{\ell}^{\top} < 0.
\end{align*}
By defining $\alpha_{{\ell}} = \epsilon + \sigma_{{\ell}}^2 C_i \mathcal{Q} C_{\ell}^{\top}$, we obtain Eq. \eqref{eq_mss_2lmi}. 
\end{proof}

In the ensuing result, an LMI-based optimization formulation is provided for the computation of mean square system norm.
\begin{lemma}
Suppose $A$ is Hurwitz and let $\theta > 0$ be a diagonal matrix. Then, we obtain, $\parallel \theta^{-1} \mathbb{G} \theta \parallel_{MS}^2 =  \inf_{\mathcal{P} > 0, \mathcal{S} > 0, \gamma} \gamma$ subject to $\mathcal{S}_{{\ell}{\ell}} < \gamma, \ {\ell} =1,2,\dots,m,$
\begin{equation*}
{
\begin{bmatrix}
A^{\top} \mathcal{P} + \mathcal{P} A & \mathcal{P} B \theta \\
\theta B^{\top} \mathcal{P} & - I
\end{bmatrix} < 0, \;\ \begin{bmatrix}
\theta \mathcal{S} \theta & C \\
C^{\top} & \mathcal{P}
\end{bmatrix} > 0.}
\end{equation*}
\label{lemma_lmi_mss}
\end{lemma}
\begin{proof}
For the system, $\theta^{-1} \mathbb{G} \theta$, the mean square exponential stability conditions can be equivalently written as, there exists a $\mathcal{Q} > 0$, such that, it satisfies the following inequalities. 
\begin{align}
%\begin{aligned}
& A\mathcal{Q} + \mathcal{Q}A^{\top} + \sum_{\ell=1}^{m} B_\ell \theta_{\ell\ell}^2 B_\ell^{\top} < 0, \tag{27} \label{eq_mi1}\\
& \gamma_\ell \theta_{\ell\ell}^2 >  C_\ell \mathcal{Q} C_\ell^{\top}, \quad \ell = 1,2,\dots, m, \tag{28} \label{eq_mi2}
%\end{aligned}
\end{align}
where $\theta_{\ell\ell}$'s are the diagonal elements of $\theta$.
The column vector, $B_\ell$'s are the columns of $B$ matrix and the row vector, $C_\ell$'s are the rows of $C$ matrix. Now, multiply on both sides of inequality \eqref{eq_mi1} by $\mathcal{Q}^{-1} := \mathcal{P}$, and writing it in compact form, we obtain
\begin{equation} \tag{29}
A^{\top} {\cal P} + {\cal P} A + {\cal P} B \theta \theta^{\top} B^{\top} {\cal P}^{\top} < 0.
\label{eq_mi3}
\end{equation}
Further, satisfying the mean square exponential stability conditions, the element wise inequalities shown in Eq. \eqref{eq_mi2} can be written in compact form as a linear matrix inequality as shown below.
\begin{align} \tag{30}
\begin{aligned}
& \theta \mathcal{S} \theta > C {\cal P}^{-1} C^{\top} \\
& \mathcal{S}_{\ell\ell} < \gamma_\ell, \quad \ell = 1,2,\dots, m.
\end{aligned}
\label{eq_mi4}
\end{align}
Now rewriting Eqs. \eqref{eq_mi3}, \eqref{eq_mi4} using Schur compliments, the computation for ${\cal H}_2$ norm problem can be written as an LMI optimization problem as shown below.
\begin{align*}
\parallel \theta^{-1} \mathbb{G} \theta \parallel_2^2 = & \inf_{\gamma_\ell, \mathcal{S}>0, \mathcal{P}>0}  \sum_\ell \gamma_\ell \\
& \text{subject to} \\
& \begin{bmatrix}
A^{\top} \mathcal{P} + \mathcal{P} A & \mathcal{P} B \theta \\
\theta B^{\top} \mathcal{P} & -I
\end{bmatrix} < 0 \\
& \begin{bmatrix} \theta \mathcal{S} \theta & C \\ C^{\top}  & \mathcal{P} \end{bmatrix} > 0 \\
& \mathcal{S}_{\ell\ell} < \gamma_\ell, \quad \ell =1,2,\dots,m.
\end{align*}
 Now, the cost is modified to obtain the result by observing the difference between $\parallel \mathbb{G} \parallel_{MS}^2 = \max_{\ell=1:m} \mathcal{S}_{\ell\ell}$ and $\parallel \mathbb{G} \parallel_2^2 = \sum_{\ell=1}^m \mathcal{S}_{\ell\ell}$.\end{proof}

The following theorem is the main result in this section and provides equivalent necessary, sufficient conditions for mean square exponential stability of feedback interconnected system \eqref{eq_system_unc}. 
In fact,  the results of the following theorem can be viewed as a stochastic counterpart of the small gain theorem for the continuous-time system.
\begin{theorem} \label{theorem_stability} Under Assumption \ref{assumptions}, consider the feedback interconnected system \eqref{eq_system_unc} as shown in Fig. \ref{fig_ip_op_unc2}b). Then, the following stability conditions for mean square exponentially stable are equivalent.
\begin{enumerate}
\item[(a)] The feedback interconnection of nominal system $\mathbb{G}$ with stochastic uncertainty, $\frac{d\Delta}{dt}$ is mean square exponentially stable.
\item[(b)] There exists a $\mathcal{Q} > 0$ and $\alpha_{\ell} > 0$, for every ${\ell} = 1,\dots,m$, satisfying the LMI given in Eq. \eqref{eq_mss_2lmi}.
\item[(c)] $\rho ( \tilde{G} \tilde{\Sigma} )<1,$ 
\vspace{0.05 in}
where $\rho$ stands for the spectral radius of a matrix and $\tilde{\Sigma} = {\rm diag} (\sigma_1^2,\cdots, \sigma_m^2)$, 
\[\mathsmaller {\tilde{G} = \begin{bmatrix}\parallel G_{11}\parallel_2^2&\ldots &\parallel G_{1m}\parallel_2^2\\\vdots&\ddots&\vdots\\\parallel G_{m1}\parallel_2^2&\ldots&\parallel G_{mm}\parallel_2^2 \end{bmatrix}.}\]
% The notation, $\parallel G_{ij}\parallel_2$ is the ${\cal H}_2$ norm of the system from disturbance input, $j$ and controlled output, $i$. 
\end{enumerate}
Further, for $\sigma_1^2 = \cdots = \sigma_m^2 = \sigma^2$,  the feedback interconnection is mean square exponentially stable if and only if \[\sigma^2 \inf_{\theta > 0, \theta-\text{diag}} \parallel \theta^{-1} \mathbb{G} \theta \parallel_{MS}^2 < 1.\]
\label{thm_mss_eq_conds}
\end{theorem}
\vspace{-0.3 cm}
\begin{proof}
$(a) \Leftrightarrow (b)$
This follows by combining the results from Theorem \ref{thm_mss_lyap_eq} and Lemma \ref{lemma_mss_Q}.\\ % under Assumption \ref{assumptions}. \\
$(b) \Leftrightarrow (c)$
This result follows by distributing the system to single input single output systems and using the spectral radius definition of nonnegative matrices discussed in \cite{Horn_John}. 

In the special case of all variances to be the same, the result follows from Lemma \ref{lemma_lmi_mss} and by choosing $\theta = {\rm diag} (\sqrt{\alpha_1}, \sqrt{\alpha_2}, \dots, \sqrt{\alpha_m}).$
\end{proof}

Here, the mean square norm is computed for the transformed system $\theta^{-1} \mathbb{G} \theta$. The scaling factor, $\theta$, ensures the mean square norm with respect to all inputs and outputs is same. Hence, for a SISO system, the scaling factor, $\theta$, does not come into play. Moreover, in the case for a SISO system, the mean square norm is equal to the standard ${\cal H}_2$ norm.
\begin{remark} The equivalent condition (c) from Theorem \ref{theorem_stability} can be used to determine the maximum tolerable variance of uncertainty $\sigma^*$ above, which the feedback interconnection will be mean square exponentially unstable. In particular, the critical $\sigma^*$ is given by
$$\sigma^*=\frac{1}{\sqrt{\inf_{\theta > 0, \theta-\text{diag}} \parallel \theta^{-1} \mathbb{G} \theta \parallel_{MS}^2}}.$$
\label{remark_critical_variance}
\end{remark}
The results derived until here provide a framework for determining the largest variance of channel uncertainty. However, the variance value itself  must be computed numerically. We do not have the analytical expression for the largest variance value expressed in terms of characteristics of the open-loop system dynamics. Later, in the simulation section, we show the variance value is a function of both the open-loop unstable poles and zeros. Therefore, in the next section, for a special class of systems, single input system with full state feedback, we show analytically, the maximum variance that can be tolerated by the nominal system with state feedback controller. 
\subsection{Fundamental limitations in a single input case}
\label{limitations}
In this section, we discuss the fundamental limitations in the mean square stabilization for a special case, single input system with full state feedback. The channel uncertainty is assumed at input side only. With single uncertainty in the feedback loop, the mean square system norm is reduced to standard ${\cal H}_2$ norm. Furthermore, using the standard results from robust control theory \cite{rotea1993_h2problem}, we know using the full state feedback measurements,  the optimal ${\cal H}_2$ performance obtained from static and dynamic controllers are the same. Hence, to find the controller giving optimal ${\cal H}_2$ norm, it is enough to restrict the search to the class of static controllers.
With some abuse of notation, we write the single-input LTI system with input channel uncertainty as follows.
\begin{equation}
\begin{array}{ll}
&\dot x=A_ox+ B \hat{u},\; \hat{u}=(\mu+\sigma \xi)v \\
&v=Kx,
\end{array}
\label{fund_lim}
\end{equation}
where $\hat u\in \mathbb{R}$, $\mu \neq 0$ and $\sigma$ are the mean and standard deviation of the white noise process, $\xi = \frac{d\Delta}{dt}$, with $\Delta$ being the standard Wiener process.  System matrix, $A_o$ correspond to the open-loop system. We now make the following assumption. 
\begin{assumption}\label{assume_unstable} Assume all the eigenvalues of $A_o$ are in the right-half plane, i.e., $-A_o$ is Hurwitz and the pair $(A_o,B)$ is stabilizable. 
% there exists a $K$, such that $A+\mu B K$ is Hurwitz.
 \end{assumption}
Since, $A_o$ has all eigenvalues on the right hand side, the stabilizability of pair $(A_o,B)$  is equivalent to controllability of $(A_o,B)$. Further, the pair $(A_o, \mu B)$ is also controllable, and there exists a stabilizing controller $K$ such that $A:=A_o+\mu B K$ is Hurwitz.
The objective here is to design a state feedback controller, so the closed-loop system is mean square exponentially stable with maximum tolerable variance, $\sigma^2_*$. 
\begin{theorem}\label{theorem_fund}
Consider the stabilization problem for single-input full state feedback LTI system with channel uncertainty at the input side shown in Eq. \eqref{fund_lim}. Under Assumption \ref{assume_unstable}, system (\ref{fund_lim}) is mean square exponentially stable, if and only if,
$$2 \frac{\sigma^2}{\mu^2} \sum_i \lambda_i(A_o) <  1.$$
\label{them_fund_lim}
\end{theorem}
\begin{proof} 
The system \eqref{fund_lim} with plant, state-feedback controller and uncertain input channel are written in closed-loop form as 
\begin{align}
\dot{x} = & A x + \sigma B K x \xi. \label{cl_fund}
\end{align}
This closed-loop system \eqref{cl_fund} can be further written as a SISO system with nominal system $\mathbb{G}$ and stochastic uncertainty, $\xi$ in the feedback. The system matrices of the nominal system, $\mathbb{G}$ are   
$ \begin{pmatrix}
A_o+\mu B K & B \\ K & 0
\end{pmatrix}.$
%
%This follows by noticing that the input ($u$) and output ($y$) are given by $u = \sigma \xi y$ and $y = K x$. 
This follows by noticing that the disturbance ($w$) and control ($z$) signals are given by $w = \xi z$ and $z = K x$ respectively.
We recall that the mean square norm for SISO system \eqref{cl_fund} is equivalent to $\mathcal{H}_2$ norm of the system \eqref{cl_fund}. Based on this, we first show the necessity part. 

\textit{Necessity:} From Theorem \ref{theorem_stability}, we know that the necessary condition for mean square exponential stability is, $\sigma^2 \parallel \mathbb{G} \parallel_2^2 < 1$.  The $\mathcal{H}_2$ norm of $\mathbb{G}$, i.e., $\parallel \mathbb{G} \parallel_2$, is given by $B^{\top} P B$ where $P > 0$, and is obtained from 
\begin{align}
\mathsmaller{
(A_o+\mu B K)^{\top} P + P (A_o+\mu B K) + K^{\top} K = 0. }
\label{eq_mss_siso}
\end{align}
Now, the optimal $K$ satisfying Eq. \eqref{eq_mss_siso} is obtained by minimizing the left hand side (lhs) of Eq. \eqref{eq_mss_siso}, i.e., by taking the derivative of lhs of Eq. \eqref{eq_mss_siso} w.r.t $K$ and equating it to zero which yields, $K = -\mu B^{\top} P$. Using this optimal $K$ in Eq. \eqref{eq_mss_siso}, we obtain, 
$$A_o^{\top} P + P A_o - \mu^2 P B B^{\top} P = 0.$$
Further, we rewrite this equation by multiplying and dividing the last term with $ \sigma^2 B^{\top} P B$ to obtain 
\begin{align}
\mathsmaller{A_o^{\top} P + P A_o - \frac{\mu^2}{\sigma^2} \frac{P B B^{\top} P}{B^{\top} P B}  \sigma^2 B^{\top} P B = 0. }
\label{eq_mss_siso1}
\end{align} 
Now, using the given relation, $\sigma^2 B^{\top} P B < 1$ from mean square exponential stability, we can rewrite Eq. \eqref{eq_mss_siso1} as 
\begin{align}
\mathsmaller{A_o^{\top} P + P A_o - \frac{\mu^2}{\sigma^2} \frac{P B B^{\top} P}{B^{\top} P B} < 0.}
\label{eq_mss_siso2}
\end{align}
Since, $P > 0$, pre and post multiplying Eq. \eqref{eq_mss_siso2} by $P^{-\frac{1}{2}}$ on both sides, we obtain, 
\begin{align}
\mathsmaller{P^{-\frac{1}{2}} A_o^{\top} P^{\frac{1}{2}} + P^{\frac{1}{2}} A_o P^{-\frac{1}{2}} - \frac{\mu^2}{\sigma^2} \frac{P^{\frac{1}{2}} B B^{\top} P^{\frac{1}{2}}}{B^{\top} P B} < 0.}
\label{eq_mss_siso4}
\end{align}
Now, taking trace on both sides of Eq. \eqref{eq_mss_siso4} and using the properties of trace, we have, $2\text{tr}(A_o) < \frac{\mu^2}{\sigma^2}$. 

\textit{Sufficiency:} It is enough to show, $\sigma^2 B^{\top} P B < 1$, where $P>0$ satisfies Eq. (\ref{eq_mss_siso}).
Consider Eq. \eqref{eq_mss_siso} and choose $K = -\mu B^{\top} P$ to rewrite Eq. \eqref{eq_mss_siso} as 
\begin{align}
\mathsmaller{A_o^{\top} P + P A_o - {\mu^2} \frac{P B B^{\top} P}{B^{\top} P B} B^{\top} P B = 0.}
\label{eq_mss_siso3}
\end{align}
Pre and post multiplying Eq. \eqref{eq_mss_siso3} by $P^{-\frac{1}{2}}$ on both sides and then taking trace on both sides, we obtain,
\begin{align}
2tr(A_o) - \mu^2 tr(B^{\top} P B) = 0. 
\label{eq_mss_siso5}
\end{align}
Given, $2 \frac{\sigma^2}{\mu^2} tr(A_o) < 1$. Then, Eq. \eqref{eq_mss_siso5} simplifies to $\mu^2 B^{\top} P B < \frac{\mu^2}{\sigma^2}$ and hence the result follows.
\end{proof}

\section{Mean Square Controller synthesis}
\label{controller}
In this section, we tackle the controller synthesis problem for the closed-loop system, $ {\cal F}(\mathbb{G},\frac{d\Delta}{dt})$. The controller is designed such that the closed-loop system can tolerate maximum uncertainty. Using part $(c)$ of Theorem \ref{thm_mss_eq_conds} and Lemma \ref{lemma_lmi_mss}, we pose the controller synthesis problem as an LMI-based optimization problem,
\[\inf_{\mathbb{K}-stab, LTI} \inf_{{\theta}>0,diag} \parallel {\theta}^{-1} {\cal F}(\mathbb{P},\mathbb{K}) {\theta} \parallel_{MS}^2.\]
Moreover, the designed $\mathbb{K}$ satisfies Assumption \ref{assumptions}a) which is, the nominal system $\mathbb{G} = {\cal F}(\mathbb{P},\mathbb{K})$ is internally stable.  

This optimization provides a robust optimal controller by searching in the space of linear time-invariant stabilizing controllers that minimizes the mean square norm. However, searching for a robust optimal controller is a nonconvex problem. This problem can be made convex by following the approach given in \cite{Scherer_LMI} along with fixing the variable $\theta$. Later, in simulations, we solve the optimization problem for the controller by keeping the variable $\theta$ constant.
The resultant controller formulation is given in the ensuing theorem.

\begin{theorem}
Given a plant $\mathbb{P}$ and for any ${\theta} > 0$, the optimization problem: \[\inf_{\mathbb{K}-stab, LTI} \parallel {\theta}^{-1} \mathbb{G} {\theta} \parallel_{MS}^2\]
is equivalent to the following LMI optimization:
\begin{align*}
& \inf_{{\textbf{X},\textbf{Y},\textbf{S},\hat{\textbf{A}},\hat{\textbf{B}}, \hat{\textbf{C}}, \gamma}} \gamma \\
& \text{subject to} \;\  S_{{\ell}{\ell}} < \gamma, \quad {\ell} = 1,2,\dots, m, \\
&  { \left[\begin{array}{ccc}
\begin{matrix}
\mathcal{L}_1(\textbf{X},\hat{\textbf{C}}) \\
\hat{\textbf{A}} + A_p^{\top}  
\end{matrix} & \begin{matrix}
\hat{\textbf{A}}^{\top} + A_p \\
\mathcal{L}_2(\textbf{Y},\hat{\textbf{B}}) 
\end{matrix} & \bigg(\begin{matrix}
0  & B_p \\
\hat{\textbf{B}} & \textbf{Y} B_p
\end{matrix}\bigg)  {\theta} \\
 {\theta^{\top}} \bigg(\begin{matrix}
0 \\
B_p^{\top} 
\end{matrix} & \begin{matrix}
\hat{\textbf{B}}^{\top}  \\
(\textbf{Y} B_p)^{\top}
\end{matrix}\bigg) &
 \begin{matrix}
{-I}
\end{matrix}
\end{array} \right]} < 0, \\
&  { \left[ \begin{array}{ccc}
{\theta} \mathcal{\textbf{S}} {\theta} & \begin{matrix}
C_p \textbf{X}  \\ \hat{\textbf{C}}
\end{matrix} & \begin{matrix}
C_p \\
0
\end{matrix} \\
\begin{matrix}
(C_p \textbf{X})^{\top}  & \hat{\textbf{C}}^{\top} \\  C_p^{\top} & 0
\end{matrix} & \begin{matrix}
\textbf{X} \\ I
\end{matrix} & \begin{matrix}
I \\ \textbf{Y}
\end{matrix}
\end{array}\right]} > 0, 
\end{align*}
where $\textbf{X}, \textbf{Y}, \textbf{S}$ are positive definite symmetric matrices of size $n \times n, n \times n, m \times m$, and $\hat{\textbf{A}}, \hat{\textbf{B}}, \hat{\textbf{C}}$ are matrices of sizes $n\times n, n \times q, d \times n$ correspondingly. Furthermore, \[\mathcal{L}_1(\textbf{X},\hat{\textbf{C}}) = A_p \textbf{X} + \textbf{X} A_p^{\top} + B_p \Lambda_I \hat{\textbf{C}}+ (B_p \Lambda_I \hat{\textbf{C}})^{\top}\] \[\text{and} \;\; \mathcal{L}_2(\textbf{Y},\hat{\textbf{B}}) = A_p^{\top} \textbf{Y} + \textbf{Y} A_p + \hat{\textbf{B}} \Lambda_O C_p + (\hat{\textbf{B}} \Lambda_O C_p)^{\top}.\]  
A feasible solution to the above optimization is a controller of the order of the plant, $\mathbb{P}$. Then, the system matrices of the controller can be uniquely obtained as follows: 
\begin{align*}
C_k = & \hat{\textbf{C}} (M^{\top})^{-1},  \\
B_k =  & N^{-1} \hat{\textbf{B}},  \\
A_k =  & N^{-1}(\hat{\textbf{A}} - \textbf{Y} A_p \textbf{X} - N B_k \Lambda_O C_p \textbf{X}  - \textbf{Y} B_p \Lambda_I C_k M^{\top})(M^{\top})^{-1},
\end{align*}
where $M, N$ are invertible matrices satisfying
$ N M^{\top} = I - \textbf{Y}\textbf{X}$. One possible choice for $N$ is $N N^{\top} = \textbf{Y} - \textbf{X}^{-1}$ and $M$, such that \[\begin{pmatrix}
\textbf{Y} & N \\ N^{\top} & I
\end{pmatrix} \begin{pmatrix}
\textbf{X} & M \\ M^{\top} & *
\end{pmatrix} = \begin{pmatrix}
I & 0 \\ 0 & I
\end{pmatrix}. \]
\label{thm_controller_ip_op}
\end{theorem}
\begin{proof}
The result follows from Lemma \ref{lemma_lmi_mss} and applying congruence transformation as shown in \cite{Scherer_LMI}.
\end{proof}

A similar result on controller synthesis in the case of a discrete-time system with uncertainty in feedback communication channels is given in \cite{scl04}. 
Furthermore, in \cite{scl04}, the author briefly mentions different ways to approach this type of nonconvex problem. One of the ways to solve the controller synthesis problem is by applying sub-optimal methods, such as the D-K iteration \cite{Dullerud_Paganini}. In this approach, first $\theta$ is fixed to solve for the controller matrices and then $\theta$ is updated by keeping the controller matrices constant. This process is continued until the update equation for $\theta$ converges. In general, this approach does not guarantee a global optimal controller, but can always provide a local optimal controller.  The D-step formulation in the D-K iteration is given as follows.
\begin{align*}
& \inf_{\theta} \qquad 1 \\
& \text{subject to}  \\
& \left[\begin{array}{ccc}
\begin{matrix}
\mathcal{L}_1({X},\hat{C}) \\
\hat{{A}} + A_p^{\top}  
\end{matrix} & \begin{matrix}
\hat{{A}}^{\top} + A_p \\
\mathcal{L}_2({Y},\hat{B}) 
\end{matrix} & \bigg(\begin{matrix}
0  & B_p \\
\hat{{B}} & {Y} B_p
\end{matrix}\bigg) {\theta} \\
{\theta}^{\top} \bigg(\begin{matrix}
0 \\
B_p^{\top} 
\end{matrix} & \begin{matrix}
\hat{{B}}^{\top}  \\
({Y} B_p)^{\top}
\end{matrix}\bigg) &
 \begin{matrix}
{-I}
\end{matrix}
\end{array} \right] < 0, \\
%& \left[ \begin{array}{ccc}
%{\theta} \mathcal{\textbf{S}} {\theta} & \begin{matrix}
%C_p \textbf{X}  \\ \hat{C}
%\end{matrix} & \begin{matrix}
%C_p \\
%0
%\end{matrix} \\
%\begin{matrix}
%(C_p \textbf{X})^{\top}  & \hat{C}^{\top} \\  C_p^{\top} & 0
%\end{matrix} & \begin{matrix}
%\textbf{X} \\ I
%\end{matrix} & \begin{matrix}
%I \\ \textbf{Y}
%\end{matrix}
%\end{array}\right] > 0, \\
& \left[ \begin{array}{cc}
\begin{pmatrix}
C_p {X} & C_p \\ \hat{C} & 0
\end{pmatrix} \begin{pmatrix}
X & I \\ I & Y
\end{pmatrix}^{-1} \begin{pmatrix}
(C_p {X})^{\top}  & \hat{C}^{\top} \\  C_p^{\top} & 0
\end{pmatrix} & \theta \\
\theta^{\top} & \mathcal{S}^{-1}
\end{array} \right] < 0.
\end{align*}
In the D-step of D-K iteration, only $\theta$ is the optimization variable.

We remark that dealing with fixed-order controllers is a difficult problem. The controller matrices can be extracted easily only when the controller is of the size of the plant. Similar formulations can be seen in \cite{scl04} and \cite{Scherer_LMI}. 

Designing a static output feedback controller in general even for a deterministic system is a hard problem hence we expect that designing a mean square stabilizing output feedback controller will be a difficult problem. However, we agree that this will be an interesting problem to study in our future research. 

In the subsection involving fundamental limitations result on a single input full state feedback case, we have designed a static state feedback controller. In this case, we applied the standard result from the robust control theory \cite{rotea1993_h2problem}, where the optimal ${\cal H}_2$ performance obtained from static and dynamic controllers are the same for systems with full state feedback.

\section{Simulation}
\label{simulation}
In this section, we consider a power network and demonstrate the application of the proposed results, especially Theorem \ref{thm_mss_eq_conds}. A WSCC 9-bus test system is chosen for the study. The WSCC 9-bus system (refer Fig. \ref{fig_9bus}) with nine buses and three generators is an equivalent representation of the Western System Coordinating Council (WSCC). 

Consider the structure preserving power network model \cite{Bergen_structure} consisting of a linearized swing equation. The algebraic states are eliminated at the load buses  to obtain the resultant dynamic model as shown in Eq. \eqref{plant}.  
The states, $x_p = \begin{bmatrix}
\delta & \omega
\end{bmatrix}^{\top}$, where $\delta=\begin{bmatrix}\delta_1&\delta_2&\delta_3\end{bmatrix}^\top$ and $\omega=\begin{bmatrix}\omega_1&\omega_2&\omega_3\end{bmatrix}^\top$ are the generator rotor angles and frequencies respectively. The state and input matrices are defined as, 
\begin{align*}
A_p = \begin{bmatrix}
0 & I \\
-M^{-1}\hat{L} & -M^{-1}D
\end{bmatrix}, \;  & B_p = \begin{bmatrix}
0 & 0 \\ 0 & M^{-1}
\end{bmatrix},
\end{align*}
where $M$ and $D$ are diagonal matrices and they represent inertia and damping values of the generators respectively. The inputs to the power network is $u_p = \begin{bmatrix}
0 & P_m - P_d
\end{bmatrix}^{\top}$ where $P_m$ is the mechanical input to the generator and $P_d$ is the active load demand. The output matrix $C_p$ is defined based on the observation, i.e., based on which bus frequencies are being monitored. Further, the matrix $\hat L$ is the Kron-reduced matrix of the Laplacian matrix $L$ and they are given by 
\[\mathsmaller{L = \begin{bmatrix}
L_{gg} & L_{gl} \\
L_{lg} & L_{ll}
\end{bmatrix};\; \hat{L} = L_{gg} - L_{gl} L_{ll}^{-1} L_{lg}.}
\]
The elements of the Laplacian matrix corresponding to the power network gives the admittance-weighted interconnections between the generators and load buses. The inertia and damping values at the generators are chosen to be $M = {\rm diag} \{160.64, 12.47, 8.047\}$ and $D = {\rm diag} \{10,10,10\}$. The load values and bus admittance values for the 9-bus system are obtained from \cite{WSCC_data} and the Kron-reduced Laplacian matrix for the considered 9-bus system is 
\begin{equation*}
\mathsmaller{\hat{L} = \begin{bmatrix}
   -4.5375 &   2.4111  &  2.4006 \\
    2.4111  & -4.8367  &  2.9096 \\
    2.4006  &  2.9096  & -4.6931
\end{bmatrix}.}
\end{equation*}
Next, suppose there is frequency deviation from the nominal value due to power imbalance in the network. Then, a control has to be applied at any one of the generators to accordingly change the generation in order  to regulate the frequency. Let's say, a new control input has to be applied at the generator 1 to achieve frequency regulation and the control input to the generator is stochastic. In this scenario, the designed controller will change the set-point as well as robust to the uncertainties. The input matrix corresponding to the generator 1 is given by 
\begin{eqnarray}
\hat{B}_p=\begin{bmatrix} 0 & 0 & 0 & 0.0062 & 0 & 0 \end{bmatrix}^{\top}
\end{eqnarray}
and clearly the pair $(A_p, \hat{B}_p)$ is stabilizable. The resultant closed loop system with plant, controller (as described in Eq. \eqref{controller}) and stochastic uncertainty in the control input to the generator 1 is given by 
\begin{align*}
\begin{bmatrix}
\dot{x}_p \\ \dot{x}_k
\end{bmatrix} = & \begin{bmatrix}
A_p & \hat{B}_p C_k \\ 
\hat{B}_k C_p & A_k
\end{bmatrix} \begin{bmatrix}
x_p \\ x_k
\end{bmatrix} + \sigma \begin{bmatrix}
0 & \hat{B}_p C_k \\ 0 & 0
\end{bmatrix} \begin{bmatrix}
x_p \\ x_k
\end{bmatrix} \xi
\end{align*}
where $\sigma > 0$ and $\xi = \frac{d\Delta}{dt}$ with $\Delta$ being the scalar Wiener process. This closed-loop system can be written as a networked system with uncertainty, and it is a single input single output system with uncertainty (in the feedback). Hence, the design of controller based on the mean square norm minimization is equivalent to the ${\cal H}_2$ norm minimization. The objective here is to understand the role of open-loop poles and zeros on the critical value of stochastic variance, $\sigma^2_*$. Three different choices of output are chosen as described in Table \ref{table1} with ${\cal H}_2$ norm minimization to determine the impact of open-loop zeros on $\sigma^2_*$. For every chosen output matrix, the pair $(A_p, C_p)$ is verified to be detectable. The eigenvalues and hence, the poles of the open-loop system are given as  follows:
%\vspace{-0.2 cm}
\begin{align}
\lambda  \in  \{ -1.6742,  -1.0927,     0.5255,    0.2683,   0, -0.1339\}.
\label{poles}
\end{align}
%\vspace{-0.2 cm}
\begin{table}[h!]
\caption{Effect of non-minimum phase zeros on critical variance ($\sigma^2_*$)}
\begin{center}
\begin{tabular}{|c|ccc|}
\hline
%Critical & \multirow{2}{*}{$y=\omega_1$}  &  \multirow{2}{*}{$y=\omega_1+\omega_2$}   &  \multirow{2}{*}{$y = \omega_2$} \\
%variance &&& \\
Output & $y=\omega_1$ & $y=\omega_1+\omega_2$ & $y = \omega_2$ \\
\hline
Critical  & \multirow{2}{*} {$0.43$} & \multirow{2}{*} {$0.37$} & \multirow{2}{*} {$1.6\times 10^{-4}$}\\   
variance &&& \\
%Single input & & & \\
open-loop  & $\{ -1.63, -0.8 $  & $\{-1.63, -0.8$ & $\{-1.8, -0.8$ \\
zeros & $ 0.41, -0.18, 0.11\}$ & $0.4, -0.17, 0.12\}$ & $0.55, -10^{-4}\}$ \\
\hline
\end{tabular}
\label{table1}
\end{center}
\end{table}
%
%\vspace{-0.2 cm}
The directions corresponding to the non-minimum phase zeros (unstable zeros) in state space needs more input energy to control and has less output energy to observe. In other words, the non-minimum phase zeros increase the phase lag of the system, and the controller must utilize extra effort to nullify its effect. Further, they increase the overall ${\cal H}_2$ norm for the closed-loop system. Hence, the uncertainty that can be tolerated by a system with non-minimum phase zeros far away from the imaginary axis is very small \cite[Chapters 5,6]{Skogestad_book}.
\begin{figure}[!htb]
    \centering
    \begin{minipage}{.45 \textwidth}
        \centering
        \includegraphics[scale = 1.34]{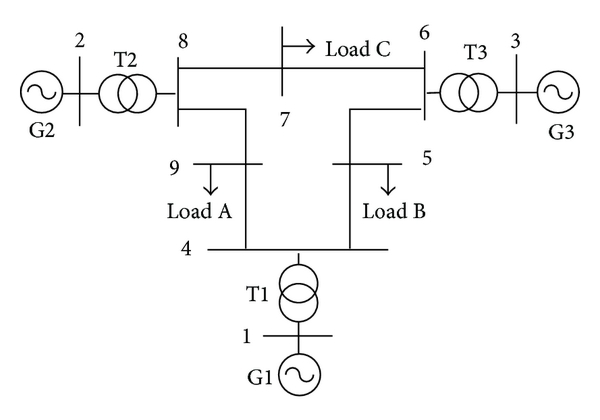}
        \caption{WSCC $9$ bus system}
        \label{fig_9bus}
    \end{minipage}%
    \begin{minipage}{0.45 \textwidth}
        \centering{
        \includegraphics[scale=0.45]{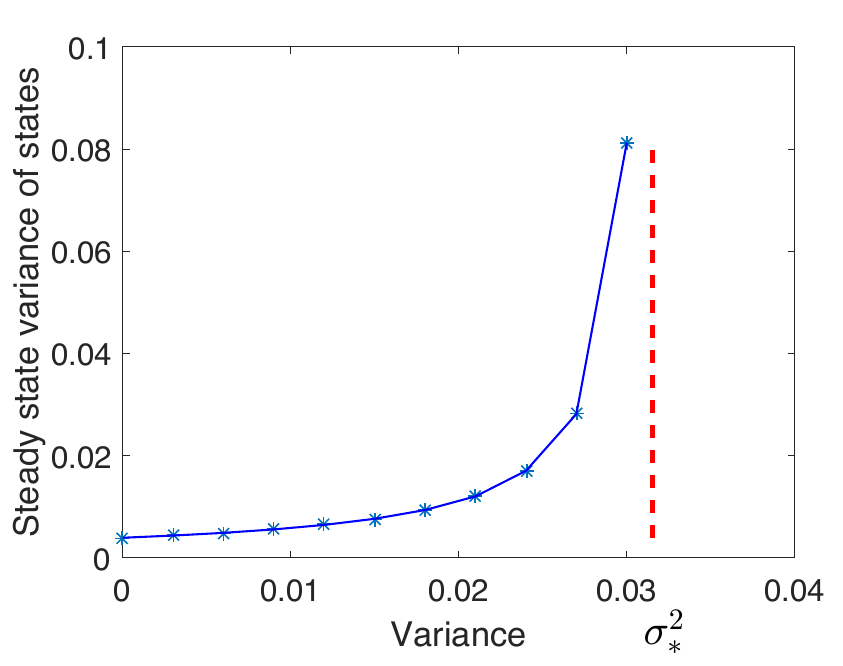}
        \caption{Variation of steady state \\ variance of states}
        \label{fig:traceQ_vs_variance}}
    \end{minipage}
\end{figure}

Next, consider the case where the controller is changing the set-point of all the generators and further it is assumed that uncertainty enters both at the input and output channels with identical variance, $\sigma^2$. The output matrix is considered such that the frequencies ($\omega$) at all the generators are observed. It is verified that the resultant system is stabilizable as well as detectable. With $\theta$ fixed, a dynamic stabilizing feedback controller is designed by solving the LMI-based optimization problem (as given in Theorem \ref{thm_controller_ip_op}) using CVX package in MATLAB. 

Now, applying Theorem \ref{thm_mss_eq_conds}, a dynamic controller is designed such that $\sigma^2_* = 0.031$. To verify the theoretical prediction for the critical value of $\sigma^2_*$, we compute the steady-state covariance for the closed-loop system for varying values of $\sigma^2$. The corresponding plot is shown in Fig. \ref{fig:traceQ_vs_variance} and observe the covariance grows unbounded as the critical value of $\sigma_*^2$ is approached. Further, a non-robust controller based on observer feedback is designed, and the critical variance of the corresponding system is $\sigma_*^2 = 0.0093$. This demonstrates the advantage of the proposed framework.

The next set of simulation results are performed to verify the fundamental limitation results. In doing so, we assume full state feedback, i.e., $C=I$ and single input at generator 3 whose corresponding $B$ matrix is 
$B=    \begin{bmatrix}     0 & 0 & 0 & 0 & 0.0802 & 0 \end{bmatrix}^{\top}.$
The critical value for $\sigma_*$, following results from Theorem \ref{them_fund_lim} is given by
$\sigma_*=\mu\sqrt{{2\sum_i \lambda_i(A)}^{-1}}.$
Assuming the mean value of uncertainty $\mu=1$ and using the pole locations from Eq. (\ref{poles}), we obtain $\sigma_*=0.793$.

\section{Conclusion}
\label{conclusion}
Necessary and sufficient conditions for mean square exponential stability of continuous-time LTI systems with input and output channel uncertainties are derived. The mean square exponential stability results are given in terms of a spectral radius condition, which includes the computation of ${\cal H}_2$ norms of the SISO deterministic systems. Further, we show the mean square exponential stability can be verified by computing the mean square system norm posed as an optimization problem using LMI's. We derive  fundamental limitation results that arise in the mean square exponential stabilization of single input LTI system with input channel uncertainty. These results generalize existing results for discrete-time linear and nonlinear system, where the limitations are expressed in terms of the eigenvalues of open-loop system dynamics. Simulation results involving network power system are presented to demonstrate the application of the developed framework. 
\bibliographystyle{ieeetran}
\bibliography{references}

\end{document}